\documentclass[fleqn,reqno,10pt]{article}
\usepackage[T1]{fontenc}
\usepackage{lmodern}
\usepackage{mathrsfs}
\usepackage{pifont}
\usepackage{microtype}
\usepackage{enumitem}
\usepackage{amsmath}
\usepackage{amsfonts}
\usepackage{amssymb}
\usepackage{mathtools}
\usepackage[matrix,arrow]{xypic}

\usepackage{xcolor}

\makeatletter
\AtBeginDocument{%
  \let\[\@undefined
  \DeclareRobustCommand{\[}{\begin{equation}}%
  \let\]\@undefined
  \DeclareRobustCommand{\]}{\end{equation}}%
}
\makeatother
\mathtoolsset{showonlyrefs,showmanualtags}

\usepackage{hyperref}
\hypersetup{
  colorlinks, 
  linkcolor=-red!75!green!50, 
  citecolor=-red!75!green!50,
  urlcolor=green!30!black
}

\usepackage[msc-links]{amsrefs}

\usepackage{amsthm}
\usepackage{thmtools}
\declaretheorem[numbered=no]{Theorem}
\declaretheorem[numbered=no]{Proposition}
\declaretheorem[numbered=no]{Corollary}
\declaretheorem[numbered=no]{Question}

\DeclareMathOperator{\Ext}{Ext}
\DeclareMathOperator{\Tor}{Tor}
\DeclareMathOperator{\pdim}{pdim}
\newcommand\place{\mathord-}
\newcommand\ZZ{\mathbb{Z}}
\newcommand\NN{\mathbb{N}}
\newcommand\id{\mathrm{id}}

\title{Untwisting algebras with van den Bergh duality\\into Calabi-Yau
algebras}

\author{Mariano Suárez-Alvarez\thanks{CONICET. Departamento de Matemática. Facultad
de Ciencias Exactas y Naturales. Universidad de Buenos Aires. Email:
\texttt{mariano@dm.uba.ar}}}

\date{October 31, 2013}

\begin{document}

\maketitle

Recently, Jake Goodman and Ulrich Kr\"ahmer \cite{GK} have shown that a
twisted Calabi-Yau algebra $A$ with modular automorphism~$\sigma$ and
dimension~$d$ can be ``untwisted,'' in the sense that the Ore extensions
$A[X;\sigma]$ and $A[X^{\pm1};\sigma]$ are Calabi-Yau algebras of
dimension~$d+1$. The purpose of this note is to record the observation that
this result holds in greater generality:

\begin{Theorem}
If $A$ is an algebra that satisfies the conditions for van den Bergh
duality of dimension~$d$ and $U=\Ext_{A^e}^d(A,A\otimes A)$ is its
dualizing bimodule, then the tensor algebra
$A[U]=\bigoplus_{n\in\NN_0}U^{\otimes_An}$ is a Calabi-Yau algebra of
dimension~$d+1$.
\end{Theorem}

If in this statement we suppose that $A$ is a twisted Calabi-Yau
algebra with modular automorphism~$\sigma$, then we recover the result of
Goodman and Kr\"ahmer, for in that case $U=A_\sigma$ is a twisted
$A$-bimodule and $A[U]$ is isomorphic to the Ore extension~$A[X;\sigma]$.
We refer to the papers of van den Bergh~\cite{vdB} and Ginzburg~\cite{G}
for the little informatiom about duality and Calabi-Yau algebras that we
need. We work over a fixed field, over which unadorned tensor products are
taken, or over an arbitrary commutative ring, provided we add the
hypothesis that $A$ be projective. All our complexes are cochain complexes,
we underline their components of degree zero and, for brevity, we say
that a complex of
$\Lambda$\nobreakdash-bimodules is \emph{good} if it is of finite length and its
components are finitely generated as bimodules.

\begin{proof}
Let us write, for simplicity, $B=A[U]$. The kernel~$I$ of the obvious
augmentation map $B\to A$ is finitely generated and projective as a
$B$-module both on the left and on the right; indeed, restriction along the
inclusion $U\hookrightarrow I$ gives an isomorphism of functors
$\hom_B(I,\place)\cong\hom_A(U,\place)$ of left or right $B$-modules, and
$U$ is finitely generated and projective as an $A$-module both on the left
and on the right. Let $K$ be the complex
  \[
  \xymatrix{
  I\otimes B \ar[r] 
        & \underline{B\otimes B}
  }\]
of $B^e$-bimodules, and let us make the convention that the left and right
actions of~$B^e$ are the outer and the inner ones, respectively; it is
clear that $K$ is a complex of projective and finitely generated right
$B^e$-modules, and that its homology is $H(K)=A\otimes B$ concentrated in
degree zero.

Let now $P$ be a good resolution of~$A$ by 
projective $A$-bimodules. The complex of right $B^e$-modules
$P\otimes_{A^e}K$ has finite length, and all its components are finitely
generated and projective. To compute its homology, we can use a spectral
sequence. Taking first homology with respect to the differential of~$P$ we
obtain ---because $K$ is a complex of projective left $A^e$-modules--- the complex
$A\otimes_{A^e}K$. This can be identified with
  \[
  \xymatrix{
  B\otimes_AI \ar[r]
    & \underline{B\otimes_AB}
  }\]
and the homology of this is~$B$ concentrated in degree
zero. We conclude in this way that $P\otimes_{A^e}K$ is a good resolution 
of~$B$ by projective $B$-bimodules.

We want to compute $\Ext_{B^e}(B,B\otimes B)$. We have
  \[
  \hom_{B^e}(P\otimes_{A^e}K,B\otimes B)
        \cong \hom_{A^e}(P,\hom_{B^e}(K,B\otimes B)) 
  \]
and, since $A$ satisfies van den Bergh duality of dimension~$d$
and dualizing module~$U$, this has the same homology as
  \[
  P\otimes_{A^e}\Bigl(U\otimes_A\hom_{B^e}(K,B\otimes B)\Bigr)[-d] .
                \label{eq:complex}
  \]
We use, as before, a spectral sequence to compute the homology of this
complex. As $U$ is finitely generated and projective as an $A$-module,
  \begin{align}
  \hom_{B^e}(I\otimes B,B\otimes B)
        & \cong \hom_{B}(I,B\otimes B)
          \cong \hom_A(U,B\otimes B) \\
        & \cong U^*\otimes_AB\otimes B,
  \end{align}
with $U^*=\hom_A(U,A)$. Now $U\otimes_AU^*\cong A$ as $A$-bimodules, so the
complex $U\otimes_A\hom_{B^e}(K,B\otimes B)$ can be identified with
  \[
  \xymatrix{
  \underline{U\otimes_AB\otimes B} \ar[r]
        & B\otimes B
  }\]
and its homology is $A\otimes B$ concentrated in degree one. It follows
that taking homology with respect to the differential induced by that
of~$K$ in the complex~\eqref{eq:complex}, we get $P\otimes_{A^e}(A\otimes
B)[-d-1]$, and the homology of this, in turn, is clearly $B$ concentrated in
degree~$d+1$. This proves the theorem.
\end{proof}

\begin{Corollary}
In the conditions of the theorem, the algebra
$C=\bigoplus_{n\in\ZZ}U^{\otimes_An}$ is also Calabi-Yau of
dimension~$d+1$.
\end{Corollary}

Notice that this makes sense because the $A$-bimodule~$U$ is invertible.

\begin{proof}
One can check at once that the multiplication in~$C$ induces an isomorphism
$C\otimes_BC\to B$. On the other hand, $C$ is flat as a left and as a right
$B$-module: it is the colimit of the chain of its $B$-submodules of the
form $\bigoplus_{n\geq n_0}U^{\otimes_An}$ with $n_0\in\ZZ$, each of which
is projective, being isomorphic to $U^{\otimes_An_0}\otimes_AB$
or~$B\otimes_AU^{\otimes_An_0}$. The corollary follows then from the
theorem and the localization result \cite{F}*{Theorem~6} of Farinati.
\end{proof}

The theorem above and its corollary are unsatisfying in that they untwist
the algebra~$A$, which is of dimension~$d$, into an algebra of
dimension~$d+1$. In~general, this seems to be as much as one can hope for.
There is a case in which we can do better, though, and we end this note
explaining this.

We put ourselves in the situation of the theorem again and suppose
additionally that $U$ is a bimodule of \emph{finite order}, so that there
exist $\ell\in\NN$ and an isomorphism $\phi:U^{\otimes_A\ell}\to A$ of
$A$-bimodules, and that the isomorphism~$\phi$ is \emph{associative}, in
the sense that the diagram
  \[
  \xymatrix@C+10pt{
  U^{\otimes_A(\ell+1)} \ar[r]^-{\phi\otimes\id_U} \ar[d]_-{\id_U\otimes\phi}
        & A\otimes_AU \ar[d]
        \\
  U\otimes_AA \ar[r]
        & U
  }
  \]
in which the unlabelled arrows are canonical isomorphisms commutes. We 
consider the $A$-submodule
  \(
  R = \{x-\phi(x):x\in U^{\otimes_A\ell}\} \subseteq B
  \).
Since the isomorphism~$\phi$ is associative, the left ideal $J=BR$ coincides
with the right ideal $RB$ and it is a bilateral ideal: we can therefore
consider the algebra $D=B/J$. There is clearly a direct sum decomposition
$D\cong\bigoplus_{i=0}^{\ell-1}U^{\otimes_Ai}$ as $A$-bimodules, and this
construction should remind us of the classical construction of cyclic
algebras over a field.

\begin{Proposition}
In the situation of the theorem, if the dualizing bimodule $U$ is of finite
order and admits an associative isomorphism $\phi:U^{\otimes_A\ell}\to A$
and $\ell$ is invertible in~$A$, then the algebra $D=A[U]/J$ with $J$
generated by $R$ as above is a Calabi-Yau algebra of dimension~$d$.
\end{Proposition}

\begin{proof}
Let $\xi\in U^{\otimes_A\ell}$ be such that $\phi(\xi)=1$; since $\phi$ is
associative, $\xi$ is central in~$B$. The ideal~$J$ is generated by~$\xi-1$
and, if~$\rho:B\to B$ is right multiplication by~$\xi-1$, the complex
  \[ \label{eq:bba}
    B \xrightarrow{\;\rho\;} \underline{B}
  \]
is a resolution of~$D$ as a left $B$-module, and we can use it to compute 
  \[ \label{eq:tor}
    \Tor_p^B(D,D) \cong
      \begin{cases}
        D, & \text{if $p\in\{0,1\}$;} \\
        0, & \text{if $p\geq2$.}
      \end{cases}
  \]
It is immediate that these isomorphisms are of left $D$-modules and by
lifting the multiplication on the right of~$D$ on~$D$ to an endomorphism of
the resolution~\eqref{eq:bba}, we see that they are actually isomorphisms of
$D$-bimodules. 

We write $\xi=\xi_1\otimes\cdots\otimes\xi_\ell$, with
$\xi_1$,~\dots,~$\xi_\ell\in U$ and omitting a sum \`a~la Sweedler, and
consider the element
  \[
    e = \tfrac1\ell
        \Bigl(
          1\otimes1
          +\sum_{r=1}^{\ell-1}
          (\xi_1\otimes\cdots\otimes\xi_r)\otimes(\xi_{r+1}\otimes\cdots\otimes\xi_\ell)
        \Bigr)
        \in D\otimes_AD.
  \]
If $\mu:D\otimes_AD\to D$ is induced by the multiplication of~$D$, then
$\phi(e)=1$ and, because of the associativity of~$\phi$, 
$de=ed$ for all~$d\in D$. It~follows that there is a $D^e$-linear
morphism $s:D\to D\otimes_AD$ such that $s(1)=e$ which splits~$\mu$ and, 
in~particular, that $D$ is a direct summand of the $D^e$-module $D\otimes_AD$;
one says in this situation that $D/A$ is a separable extension
of algebras, as in~\cite{P}*{\S10.8}. If~now $P$ is a
good resolution of~$A$ by projective $A$-bimodules,
then $D\otimes_AP\otimes_AD$ is a good resolution of~$D\otimes_AD$ by
$D$-bimodules and, since $D$ is a direct summand of~$D\otimes_AD$, we see
that $D$ itself has a good resolution by projective $D$-bimodules.

The construction done by Cartan and Eilenberg
in \cite{CE}*{\textsc{XVI}, \S5, Eq.~$(2)_3$} specialized for
the obvious morphism $B^e\to D^e$, together with the natural isomorphism
$\Tor^{B^e}(D^e,B)\cong\Tor^B(D,D)$ of \cite{CE}*{\textsc{IX}, Prop.~4.4},
gives us a change-of-rings spectral sequence with
$E_2^{p,q}=\Ext_{D^e}^p(\Tor^B_q(D,D),D\otimes D)$ converging to
$\Ext_{B^e}(B,D\otimes D)$. Since $B$ is Calabi-Yau of
dimension~$d+1$,
  \[
    \Ext_{B^e}^r(B,D\otimes D)
      \cong\Tor^{B^e}_{d+1-r}(B,D\otimes D)
      \cong\Tor^{B}_{d+1-r}(D,D),
  \]
so that we know the limit of the spectral sequence from~\eqref{eq:tor}.
From~\eqref{eq:tor} we also know that $E_2^{p,q}=0$ if $q\not\in\{0,1\}$,
and that $E_1^{p,0}\cong E_1^{p,1}\cong\Ext_{D^e}^p(D,D\otimes D)$ for
all~$p$. A standard argument with the spectral sequence ---using the fact
that $\pdim_{D^e}D<\infty$--- shows now that $D$ is Calabi-Yau
of dimension~$d$.
\end{proof}

An easy and probably important remark to be made is that the dualizing
bimodule~$U$ for an algebra~$A$ with van den Bergh duality is always
\emph{central}: the actions of the center~$Z(A)$ of~$A$ on the left and on
the right on~$U$ coincide. It~follows from this that when $U$ is of finite
order, so that there is an isomorphism of bimodules $\phi:U^{\otimes_Ae}\to
A$, the fact that $\phi$ be associative or not does not depend on the
particular choice of~$\phi$: it is a property of~$A$. It is natural to
ask:

\begin{Question}
If the dualizing bimodule of an algebra with van den Bergh duality is of
finite order, is it necessarily associative?
\end{Question}

If the algebra is twisted Calabi-Yau, the answer is affirmative.


\end{document}